\documentclass[11pt,reqno]{amsart}
\usepackage{amssymb,latexsym}
\usepackage{graphicx}
\usepackage{fancyhdr,amssymb}
\usepackage{url}	
\usepackage{amsthm}

\theoremstyle{plain}
\numberwithin{equation}{section}
\newtheorem{thm}{Theorem}[section]
\newtheorem{theorem}[thm]{Theorem}
\newtheorem{lemma}[thm]{Lemma}

\begin{document}
\fancyhead{}
\renewcommand{\headrulewidth}{0pt}
\fancyfoot{}
\fancyfoot[LE,RO]{\medskip \thepage}
\

\setcounter{page}{1}

\title[A new proof of Vantieghem's theorem.]{A new proof of Vantieghem's theorem.}
\author{Konstantinos Gaitanas}
\address{National Technical University of Athens,  School of Applied Mathematical and Physical Sciences}
\email{raffako@hotmail.com}

\begin{abstract}
We present a new proof of a primality criterion first proved by Emmanuel Vantieghem.
\end{abstract}

\maketitle

\section{Introduction}
E. Vantieghem has proved\cite{Van} that  $p>2$ is prime if and only if $\displaystyle \prod\limits_ {n=1}^{p-1}(b^n+1)\equiv 1 \pmod{\frac{b^p-1}{b-1}}$.\\

His proof was based on the following lemma proved also by him.   
\begin{lemma}
 \textbf{(Vantieghem)} Let $m$ be a natural number greater than $1$ and let $\Phi_m(X)$ be the $m^{th}$ cyclotomic polynomial.Then \\
\begin{center}$\displaystyle \prod_{\substack{1 \le d \le m, \\(d,m)=1}}(X-Y^d)\equiv\Phi_m(X)\pmod{\Phi_m(Y)}$ in $\displaystyle \mathbb Z[X,Y]$
\end{center}
\end{lemma}
We will prove the \emph{if} case of Vantieghem's theorem without the use of cyclotomic polynomials.Our proof requires only Fermat's Little theorem and some basic facts from the theory of congruences.
\section{main theorem}

\begin{theorem}
Let $b$ be a natural number with $2\leq b\leq p-1$.Then  if $p>2$ is prime $$\displaystyle \prod\limits_ {n=1}^{p-1}(b^n+1)\equiv 1 \pmod{\frac{b^p-1}{b-1}}\quad (1)$$.\\
\end{theorem}

\begin{proof}
Let $p$ be an odd prime , $r$ be the order of $2$ mod $p$ and $P=\{1,2,\ldots,p-1\}$.\\
 We will split the proof into two cases for the convience of the reader.\\
\\
\textbf{Case 1.} $r=p-1$. \\
\\
This means for every $n\in P$, $n\equiv 2^m\pmod{p},0\leq m\leq p-1$.\\
It is easy to see that if $n\equiv 2^m\pmod{p}\Rightarrow b^n+1\equiv b^{2^m}+1\pmod{\frac{b^p-1}{b-1}}$\\
We can see that after rearranging the factors in the left hand side of (1) we get 
$$\displaystyle \prod\limits_ {n=1}^{p-1}(b^n+1)\equiv \displaystyle \prod\limits_ {m=1}^{p-1}(b^{2^m}+1)\equiv(b^{1}+1)\cdot (b^{2^1}+1)\cdots (b^{2^{p-2}}+1)\equiv \frac{b^{2^{p-1}}-1}{b-1}\pmod{\frac{b^p-1}{b-1}} $$.\\
From Fermat's Little theorem we know that $2^{p-1}\equiv 1\pmod{p}\Rightarrow b^{2^{p-1}}\equiv b\pmod{\frac{b^p-1}{b-1}}\Rightarrow$\\
$$\frac{b^{2^{p-1}}-1}{b-1}\equiv1\pmod{\frac{b^p-1}{b-1}}$$\\
This means $\displaystyle \prod\limits_ {n=1}^{p-1}(b^n+1)\equiv 1\pmod{\frac{b^p-1}{b-1}}$ and the first case is proved.\\
\\
\\
 \textbf{Case 2.}$r<p-1$.\\
\\
This means that the numbers $1,2^1,\ldots, 2^{r-1}$ are incogruent $\pmod{p}$ and from Fermat's little theorem we know that $r\mid{p-1}$.\\
We will split the set $P=\{1,2,\ldots,p-1\}$ into $k=\frac{p-1}{r}$ subsets in the following way:\\
  
Let $A_1=\{1,2^1,\ldots ,2^{r-1}\}$be the first set and $a_i\in P$ be the smallest integer that is not contained in any of the sets $A_1,\ldots,A_{i-1}$.\\
Then $A_i=\{ a_i\cdot 1,a_i\cdot 2^1,\ldots ,a_i\cdot 2^{r-1}\}$. \\
\\
We shall prove that if the elements of the subsets are reduced modulo $p$ then\\
 $A_1\cup A_2\ldots \cup A_k=P$ and it suffices to prove that all the elements of the sets are pairwise incogruent modulo $p$.\\

If two elements belong in the same set $A_i$, suppose that  $a_i\cdot 2^m\equiv a_i\cdot 2^n \pmod{p}$ with $n<m$.\\
 Since  $p\nmid a_i$ we obtain $ 2^{n}\equiv 2^m\pmod{p}$ which leads to a contradiction since by definition the numbers $1,2,\ldots,2^{r-1}$ are all incogruent modulo $p$.\\
\\
We consider now the case when two elements belong to different sets.\\
\\
Suppose that $a_j\cdot 2^m\equiv a_i\cdot 2^n \pmod{p}$ , $1\leq m,n\leq r-1$ and without loss of generality  $i< j$.\\
Multiplying both sides with $2^{r-m}$ yields $a_j\cdot 2^r\equiv a_i\cdot 2^{r+n-m} \pmod{p}\Rightarrow a_j\equiv a_i\cdot 2^{r+n-m} \pmod{p}$. \\
 But this means that $a_j\in A_i=\{a_i\cdot 1,\ldots,a_i\cdot 2^{r-1}\}$ ,which is a contradiction since $a_j$ is by definition the smallest integer not belonging in any of the sets $A_1,\ldots,A_i,\ldots,A_{j-1}$. \\  
This means that every natural number not greater than $p-1$ is an element in its reduced form in exactly one of the sets $A_i$, $1\leq i\leq k$, which yields   $A_1\cup A_2\ldots \cup A_k=P$.\\
\\
This means for every $n\in P$, $n\equiv a_i\cdot 2^m\pmod{p},0\leq m\leq r-1$.\\  
So, $b^n\equiv b^{{a_i}\cdot 2^m}\pmod{\frac{b^p-1}{b-1}}$ and we can obtain that
$$\prod\limits_ {n=1}^{p-1}(b^n+1)\equiv \prod\limits_{i=1}^{\frac{p-1}{r}}\cdot \prod\limits_{m=0}^{r-1}(b^{{a_i}\cdot 2^m}+1)\pmod{\frac{b^p-1}{b-1}}$$
		
But we can see that $$\prod\limits_{m=0}^{r-1}(b^{a_i\cdot 2^m}+1)=((b^{a_i})^1+1)((b^{a_i})^{2^1}+1)\cdots ((b^{a_i})^{2^{r-1}}+1)=\frac{(b^{a_i})^{2^r}-1}{b^{a_i}-1}$$.\\
 Since $2^r\equiv 1\pmod{p}$ and $p\nmid a_i$ ,  $(b^{a_i})^{2^r}-1\equiv b^{a_i}-1\pmod{\frac{b^p-1}{b-1}}\Rightarrow\frac{(b^{a_i})^{2^r}-1}{b^{a_i}-1}\equiv1\pmod{\frac{b^p-1}{b-1}}$.\\
This means  $\prod\limits_{m=0}^{r-1}(b^{a_i\cdot 2^m}+1)\equiv1\pmod{\frac{b^p-1}{b-1}}$ and we can obtain immediatelly:\\

$$\prod\limits_ {n=1}^{p-1}(b^n+1)\equiv \prod\limits_{i=1}^{\frac{p-1}{r}} 1\equiv1^{\frac{p-1}{r}}\equiv 1 \pmod{\frac{b^p-1}{b-1}}$$
\begin{center}
This completes the proof.
\end{center}
\end{proof}
\section{numerical examples}
Let $p=89$ and $b=2$. The order of $2$ modulo $89$ is $r=11$. \\
 \begin{center}
The subsets from our proof are
\end{center}
\begin{center}$A_1 =\{1,2,4,8,16,32,64,39,78,67,45\}$ \\
  $A_2 =\{3,6,12,24,48,7,14,28,56,23,46\}$ \\
\end{center} 
\begin{center}$A_3 =\{5,10,20,40,80,71,53,17,34,68,47\}$ \\
$A_4=\{9,18,36,72,55,21,42,84,79,69,49\}$ \\
\end{center}
\begin{center} $A_5 =\{11,22,44,88,87,85,81,73,57,25,50\}$\\
$A_6 =\{13,26,52,15,30,60,31,62,35,70,51\}$\\
\end{center}
\begin{center} $A_7 =\{19,38,76,63,37,74,59,29,58,27,54\}$\\
$A_8 =\{33,66,43,86,83,77,65,41,82,75,61\}$
\end{center}
The numbers $a_2=3,a_3=5,a_4=9,a_5=11,a_6=13,a_7=19$ and $a_8=33$ are the least natural numbers not greater than $89$ not appearing in any of the previous subsets $A_1,A_2,A_3,A_4,A_5,A_6,A_7$ and $A_8$ respectively.\\
We can verify by brute force that $(2^1+1)(2^2+1)(2^3+1)\cdots (2^{88}+1)\equiv 1\pmod{2^{89}-1}$

\medskip

\noindent MSC2010: 11A07, 11A41

\end{document}